\documentclass[10pt]{article}
\usepackage[letterpaper, margin=3cm]{geometry}

\usepackage{color, amsmath,amssymb, amsfonts, amstext,amsthm, latexsym}
\usepackage[none]{hyphenat}
\usepackage{epsfig}
\usepackage{longtable}
\usepackage{graphicx}
\usepackage{subfigure}

\newtheorem{theorem}{Theorem}
\newtheorem{lemma}{Lemma}
\newtheorem{definition}{Definition}
 \newtheorem{coro}{Corollary}
 
  \newtheorem{Main Results}[theorem]{MainResults}

 \newtheorem{remark}{Remark}
  \newtheorem*{assumption}{Assumption}

\begin{document}

\title{Governing equations for Probability densities of stochastic differential equations with discrete time delays}
 \author{ {\it{Yayun Zheng\textsuperscript{1}  and Xu Sun\textsuperscript{1,2 ,\footnote{Corresponding author: xsun15@gmail.com}}}}\\ \\
{\textsuperscript{1}School of Mathematics and Statistics,}\\{Huazhong University of Science and Technology,
  Wuhan 430074, Hubei, China} \\
   \\{\textsuperscript{2}Department of Applied Mathematics,}\\ {Illinois Institute of Technology,
  Chicago, IL 60616, USA}\\ \\
  }

\date{Sep. 5th, 2016}
\maketitle

\begin{abstract}
The time evolution of probability densities for solutions to  stochastic differential equations (SDEs) without delay is usually described by   Fokker-Planck equations, which require  the adjoint of the infinitesimal generator for  the solutions.  However, Fokker-Planck equations  do not exist for stochastic delay differential equations (SDDEs) because  the solutions to SDDEs are not Markov processes and  have no corresponding infinitesimal generators. In this paper,  we address the open question of finding the governing equations  for probability densities of SDDEs with discrete time delays.   The governing equation is given in a simple form that facilitates theoretical analysis and numerical computation.  An illustrative example is presented to verify the proposed governing equations.

\textbf{Keywords}: stochastic  differential equations, Brownian motions, probability density, discrete time delay,  stochastic delay  differential equations.

\end{abstract}

\maketitle

\pagestyle{plain}

\section{Introduction}  \label{intro}

We shall consider the following stochastic delay differential equation (SDDE),
\begin{align}\label{s1_1}
\begin{cases}
{\rm d}X(t) = f(X(t), X(t-\tau)) {\rm d}t + g(X(t), X(t-\tau)){\rm d}B(t), \quad \text{for} \quad t>0,\\
X(t)=\gamma(-t), \quad \text{for} \quad -\tau \le t\le 0,
\end{cases}
\end{align}
where $X(t)$ is a $\mathbb{R}^d$-valued stochastic process, $B(t)$ is a $\mathbb{R}^n$-valued Brownian motion defined on some probability space $(\Omega, \mathcal{F}, \mathbb{P})$, $f: \mathbb{R}^{2\times d}\to \mathbb{R}^d$, $g: \mathbb{R}^{2\times d}\to \mathbb{R}^{d\times n}$ and $\gamma : [0, \tau]\to \mathbb{R}^d$.

SDDE (\ref{s1_1}) have  been extensively used in many fields such as biology \cite{BeuterBelairLabrie1993}, mechanical engineering \cite{CrawfordVerriestLieuwen2013}, control systems \cite{GuZhu2014}, and so on.

Sufficient conditions for  existence and uniqueness  of the solution $X(t)$ defined by (\ref{s1_1}) have been established under global Lipschitz or under local Lipshitz and linear growth conditions in the general context where the coefficients of the equation depend on the past path of the solution,  see \cite{Mao1997,Mohammed1984}  among others. Existence and regularity of the densities in the general context has been studied by the method of Malliavin calculus in \cite{KusuokaStroock1982} under some H\"{o}rmander conditions. A more general sufficient condition  is presented in \cite{BellMohammed1991} for the SDDE (\ref{s1_1}) without drift terms (i.e., $f=0$).

Governing equations for  probability densities of solutions to SDEs without delay (e.g., $f(X(t), X(t-\tau))=f(X(t))$ and $g(X(t), X(t-\tau))=g(X(t))$ in (\ref{s1_1})) are well known as Fokker-Planck equations, which have been widely used to quantify the evolution and propagation of the uncertainty in stochastic dynamical systems \cite{Duan2015,LinCai2004}. Fokker-Plank equations require the adjoint of  infinitesimal generators  for solutions  to   SDEs. However, due to its non-Markov property,  SDDE (\ref{s1_1}) has no  infinitesimal generator and thus has no corresponding Fokker-Planck equation. It is still an open question on how to obtain  governing equations for the density associated with SDDE (\ref{s1_1}). 

Note that governing equations  are often necessary to devise   analytic or numerical methods (other than Monte Carlo) to solve  the densities. The unavailability of the governing equations for the densities poses as a significant obstacle on the application of SDDE (\ref{s1_1}).

The main objective of this paper is to derive an governing equation for the probability density of the solution $X(t)$ defined by SDDE (\ref{s1_1}). The sections of this paper are organized as follows. In section 2, the main result  and its corollary are presented. Proof of the main result  is presented in section 3. In section 4, the main result is verified by an illustrative example.

\section{Main Results}

To study SDDE (\ref{s1_1}), we associate it with the following stochastic differential equation (SDE),
\begin{align}\label{s1_2}
\begin{cases}
{\rm d}X_{1}(t' )= f(X_1(t^\prime ),\gamma(\tau-t' )){\rm d}t' + g(X_1(t' ),\gamma(\tau-t' )){\rm d}B_1(t' ),\\
{\rm d}X_{2}(t' )= f(X_2(t' ),X_1(t' )){\rm d}t'  + g(X_2(t' ),X_1(t' )){\rm d}B_2(t'),\\
\vdots\\
{\rm d}X_{k}(t' )= f(X_{k}(t' ),X_{k-1}(t' )){\rm d}t'  + g(X_{k}(t' ),X_{k-1}(t^\prime)){\rm d}B_{k}(t'), 
\end{cases} \quad\text{for}\quad t' \in [0, \tau],
\end{align}
where $k\in \mathbb{N}$,  $ {X} _i(t')$ ($i=1,2, \cdots,k$) is a $\mathbb{R}^d$-valued stochastic process, $B_i(t')$($i=1, 2, \cdots, k$) is a $\mathbb{R}^n$-valued Brownian motion, which is related to $B(t)$ in (\ref{s1_1}) by $B_i(t')=B(t'+(i-1)\tau)-B_i((i-1)\tau)$. It is obvious  that $B_i(t')$ ($i=1, 2, \cdots,k)$ are  independent of each other in the probability space $(\Omega, \mathcal{F}, \mathbb{P})$.

In this paper, (\ref{s1_2}) will be investigated under three different types of  constraints, as listed below.   To simplify notation, we introduce  $\widetilde{\mathbf{X}}_k(t^\prime)$, which is a vector  in $\mathbb{R}^{k\times d}$ defined as the concatenation of the $k$ vectors $X_1(t^\prime),X_2(t^\prime),\cdots,X_k(t^\prime)$, i.e., $\widetilde{\mathbf{X}}_k(t^\prime) = (X_1 (t^\prime), X_2(t^\prime),\cdots,X_k (t^\prime)) $. 
\\ \\
 (C1) Initial condition\\ 
In this type of condition, the initial value for (\ref{s1_2})  is prescribed, i.e.,  
\begin{align}\label{s1_3}
\widetilde{\mathbf{X}}_k(0) =  v_0,
\end{align}
 where $v_0$ is a constant in $\mathbb{R}^{k\times d}$.\\ \\
(C2) Bridge condition\\ 
 In this type of condition, both the initial and  final values in the time interval $[0, \tau]$ are prescribed, i.e.,
\begin{align}\label{s1_4}
\widetilde{\mathbf{X}}_k(0) =v_0, \quad\quad \widetilde{\mathbf{X}}_k(\tau)  =v_1.
\end{align}
 where $v_0$ and $v_1$ are constants  in $\mathbb{R}^{k\times d}$.\\ \\
(C3) Continuous condition\\ 
In this type of condition, the initial value of $X_1(t')$ is prescribed, and  the initial value of $X_i(t')$ is set to be equal to the final value of $X_{i-1}(t')$ for $i=2,3, \cdots,k$. i.e., 
\begin{align}\label{s1_5}
X_1(0) = x_0\quad \text{and}\quad X_{i}(0) = X_{i-1}(\tau) \quad\text{for}\quad i=2, 3, \cdots,k,
\end{align}
where $x_0$ is a constant in $\mathbb{R}^d$ .

The following assumption   is used throughout this paper.
\begin{assumption}[H1]
Suppose $\forall v_0\in \mathbb{R}^{k\times d}$, the SDE defined by (\ref{s1_2})  and (\ref{s1_3})   have  unique strong solution, and the probability density for this solution exists and is strictly positive.
\end{assumption}

Sufficient conditions for the existence and uniqueness for the solution to SDE (\ref{s1_2}) and (\ref{s1_3}) have been thoroughly studied and are well known. To ensure the existence and uniqueness, functions $f$ and $g$ are usually required to satisfy some Lipschitz or H\"{o}lder continuous conditions.  Readers are referred to the monographs \cite{Protter2004,Klebaner2005} among others for more discussion on this topic.

The existence and regularity of the density for  the solution to   SDE (\ref{s1_2}) and (\ref{s1_3}) have   been  well studied. The coefficient $g$  often requires to satisfy some ellipticity conditions or H\"ormander conditions to ensure the existence and regularity of the density,  see \cite{Nualart2006,BogachevKrylovRocknerShaposhnikov2015} and the reference therein for more details. The strictly positive property of the densities for a general class of SDEs can be concluded from the heat kernel estimations \cite{Davies1989,Aronson1968,NorrisStroock1991}.  A more general sufficient condition for strictive positiveness of  densities is recently presented in \cite{BogachevRocknerShaposhnikov2009}.

\begin{definition}\label{def_1}
For $k\in \mathbb{N}$, define $\mathcal{Q}_{k}(\cdot;\cdot\big| \cdot;\cdot): \mathbb{R}^{k\times d }\times [0,\tau] \times \mathbb{R}^{k\times d} \times [0,\tau]\to [0, \infty)$ such that  $\mathcal{Q}_{k }(\cdot; t^\prime\big| v;s): \mathbb{R}^{k\times d } \to [0, \infty)$ ($\tau\ge t^\prime> s\ge 0$) is the  probability density for the solution $\widetilde{\mathbf{X}}_k(t^\prime)$ as defined in  SDEs (\ref{s1_2}) with initial value $
\widetilde{\mathbf{X}}_k(s)=v$.
\end{definition} 

\begin{remark}It is naturally true that
\begin{align}\label{s1_6}
  \lim\limits_{t^\prime\to s}\mathcal{Q}_{k }(u; t^\prime\big| v;s) = \delta(u-v).
\end{align}
\end{remark}
Note that three different notations are used in this paper to represent probability densities: $\mathcal {P}_\mathcal{A}$, $\mathcal{Q}_k$, and $p$. \\ \\ 
(i)$\mathcal P_\mathcal{A}$: $\mathcal{P}_\mathcal{A}$ is reserved to denote the density  for the solution $X(t)$ defined in (\ref{s1_1}). Here the subscript $\mathcal{A}$ is used to indicate the initial condition $X(s)=\gamma(-s)$ for $s\in [-\tau,0]$.  For example, $\mathcal{P}_\mathcal {A}(x,t)$ represents the density for $X(t)$ at $X(t)=x$, $\mathcal{P}_\mathcal {A}(x,3\tau; y,4\tau  \big| m, \tau; z, 2\tau )$ represents the conditional   density of $X(3\tau)$ and $X(4\tau)$ at $X(3\tau)=x$ and $X(4\tau)=y$  given $X( \tau)=m$ and $X(2\tau)=z$.\\ \\
(ii) $\mathcal {Q}_k$: As given in Definition \ref{def_1}, $\mathcal{Q}_k$ is reserved to denote the transitional density of the $\mathbb{R}^{k\times d}$-valued solution $\widetilde{\mathbf{X}}_k$ defined by SDE (\ref{s1_2}). For example,  $\mathcal{Q}_2 ( x, y;  t^\prime\big| m, z; s)$ with $0\le s< t^\prime\le\tau$ represents the density of $( X_1(t^\prime),X_2(t^\prime)) $  at $X_1(t^\prime)=x$ and $X_2(t^\prime)=y$ given $X_1(s)=m$ and $X_2(s)=z$. \\ \\
(iii) $p$: $p$ is used in  general cases including the cases where $\mathcal{P}_{\mathcal A}$ and $\mathcal{Q}_k$ do not apply. For example, $p(X=x; Y=y)$ represents the   density of the $(X,Y)$ at $X=x$ and $Y=y$; $p(X=x; Y=y\big| Z=z)$ represents the   density of $(X,Y)$ at $X=x$ and $Y=y$ given $Z=z$. Note that $\mathcal P_\mathcal{A}$ and  $\mathcal {Q}_k$ can also be expressed in terms of $p$, for instance,
\begin{align}\label{s1_7}
\mathcal{P}_\mathcal {A}(x,t)=p(X(t)=x\big| X(0)=\gamma_0),
\end{align}
\begin{align}\label{s1_8}
\mathcal{P}_\mathcal {A}(x,3\tau; y,4\tau  \big| m, \tau; z, 2\tau )=p(X(3\tau)=x; X(4\tau)=y\big| X(0)= \gamma_0,  X(\tau)=m; X(2\tau)=z),
\end{align}
\begin{align}\label{s1_9}
\mathcal{Q}_2 ( x, y; t^\prime\big| m, z;s)
&=p(X_1(t^\prime)=x; X_2(t^\prime)=y\big| X_1(s)=m; X_2(s)=z)\nonumber\\
&=p(\widetilde{\mathbf{X}}_2(t^\prime)=(x,y) \big|\widetilde{\mathbf{X}}_2(s)=(m,z) ).
\end{align}
Note that  $\gamma_0$ in (\ref{s1_7}) and (\ref{s1_8}) is a shorthand notation for $\gamma(0)$.

We are now ready to present the main result.
\begin{theorem}[Main result]\label{theorem1}
Suppose that Assumption  $H1$ holds. Then   $\forall t >0$,  the probability density function $\mathcal{P}_\mathcal{A}(x,t)$  for the solution $X(t)$  defined by (\ref{s1_1}) exists. Moreover, $\forall x\in \mathbb{R}^d$, the following statements are true: \\
(i)   For $t\in (0,\tau]$, 
\begin{align}\label{s1_10}
{\mathcal P}_{\mathcal{A}}(x,t) =  
 \mathcal{Q}_1 (x; \tau\big| \gamma_0; 0),
\end{align} 
(ii) For $t=k\tau$ with $k\ge 2$ and $k\in \mathbb{N}$,
\begin{align}\label{s1_11}
{\mathcal P}_{\mathcal{A}}(x,t)=  
\displaystyle\int_{\mathbb{R}^{(k-1)\times d}} \mathcal{Q}_k( x_1, x_2, \cdots,x_{k-1}, x; \tau\big|  \gamma_0, x_1, x_2, \cdots,x_{k-1};0)  \prod\limits_{i=1}^{k-1}{\rm d}x_i, 
\end{align}
(iii)  For $t\in ((k-1)\tau, k\tau)$ with $k\ge 2$ and $k\in \mathbb{N}$,
\begin{align}\label{s1_12}
{\mathcal P}_{\mathcal{A}}(x,t)&= 
 \displaystyle\int _{\mathbb{R}^{2\times (k-1)\times d}}  \mathcal{Q}_{k-1}( y_1, y_2, \cdots, y_{k-1}; \tau\big|  x_1, x_2, \cdots, x_{k-1} ; t-(k-1)\tau)  \nonumber \\ 
  &\quad  \times \mathcal{Q}_k ( x_1, x_2, \cdots, x_{k-1}, x; t-(k-1)\tau \big|  \gamma_0, y_1, y_2, \cdots, y_{k-1}; 0)  \prod\limits_{i=1}^{k-1}{\rm d}x_i\prod\limits_{i=1}^{k-1}{\rm d}y_i,
\end{align}
\end{theorem}
\begin{remark}
In theorem \ref{theorem1}, the density for  SDDE (\ref{s1_1}) is expressed  in terms of that for SDE (\ref{s1_2}). The latter density is well studied and can be  obtained by solving the corresponding Fokker-Planck equation.
\end{remark}

\begin{remark}\label{remark3}
$\forall u,v\in \mathbb{R}^k$, define  $\mathcal{Q}_k(u,s\big| v,s) = \lim\limits_{t\to s} \mathcal{Q}_k(u,t\big| v,s)=\delta(u-v)$,  and by using $f(x)= \int_{\mathbb{R}^d} \delta(x-y)f(y){\rm d}y $, equations  (\ref{s1_11}) and (\ref{s1_12}) can be combined together, and then equations (\ref{s1_10}), (\ref{s1_11}) and (\ref{s1_12}) can be written formally as
\begin{align}\label{s1_13}
{\mathcal P}_{\mathcal{A}}(x,t)=  
\begin{cases}
\mathcal{Q}_1 (x; \tau\big| \gamma_0; 0),\quad\text{for}\quad t\in (0, \tau],\\ \\  \\
 \displaystyle\int _{\mathbb{R}^{2\times(k-1)\times d}}\left[ \mathcal{Q}_{k-1}( y_1, y_2, \cdots, y_{k-1}; \tau\big|   x_1, x_2, \cdots, x_{k-1} ; t-(k-1)\tau)\right. \\ 
  \quad \left.\times \mathcal{Q}_k ( x_1, x_2, \cdots, x_{k-1}, x; t-(k-1)\tau \big|  \gamma_0, y_1, y_2, \cdots, y_{k-1}; 0)\right] \prod\limits_{i=1}^{k-1}{\rm d}x_i\prod\limits_{i=1}^{k-1}{\rm d}y_i, \\
  \quad\text{for }\quad  t\in((k-1)\tau, k\tau],k\in\mathbb{N}, k\ge 2.
\end{cases}
\end{align}
\end{remark}

For SDE (\ref{s1_2}), if  $f$ and $g$ are Lipshitz continuous, $f\in C^1(\mathbb{R}^{2\times d}, \mathbb{R}^d)$, $g\in C^2(\mathbb{R}^{2\times d}, \mathbb{R}^{d\times n})$, $\gamma\in C^1(\mathbb{R}, \mathbb{R}^d)$ and $\forall  x,y \in \mathbb{R}^d$, there exists a constant $\epsilon$ such that $g(x,y)\ge \epsilon>0$, then, as shown in \cite{BogachevRocknerShaposhnikov2009},  Assumption $H1$ holds. Moreover,  $\forall x_i\in \mathbb{R}^d$ ($i=1,2,\cdots,k$), the transitional density $\mathcal{Q}_k$  satisfies the celebrated Fokker-Planck equation 
\begin{align}\label{s1_14}
\begin{cases}
  \dfrac{\partial}{\partial t^\prime}\mathcal{Q}_{k}(x_1, x_2, \cdots,x_k; t^\prime\big|v,s)  = -\sum\limits_{i=1}^{k} \mathbf{\nabla}_{x_i}\cdot\left(F_i(x_1,x_2,\cdots,x_k,t^\prime)\mathcal{Q}_{k}(x_1, x_2, \cdots,x_k,t^\prime\big|v,s)\right)\\
\quad\quad+ \dfrac{1}{2}\sum\limits_{i=1}^k  {\rm {Tr}}\big\{ \mathbf{\nabla}_{x_i}\mathbf{\nabla}_{x_i} ^{T}\left(G _i(x_1,x_2,\cdots,x_k,t^\prime) G _i^T (x_1,x_2,\cdots,x_k,t^\prime) \mathcal{Q}_{k}\left(x_1, x_2,\cdots,x_k,t^\prime\big|v,s\right)\right)\big\}, \\ \quad \text{for}\quad \tau\ge t^\prime>s\ge 0,\\ \\
\lim\limits_{t^\prime\to s} \mathcal{Q}_{k}(x_1, x_2, \cdots,x_k; t^\prime\big|v,s) = \delta ((x_1,x_2,\cdots,x_k)-v),
\end{cases}
\end{align}
where $F_1(x_1, x_2, \cdots,x_k,t^\prime)=f(x_1, \gamma(\tau-t^\prime))$, $G_1(x_1, x_2, \cdots,x_k,t^\prime)=g(x_1, \gamma(\tau-t^\prime))$,  $F_i(x_1, x_2, \cdots,x_k, t^\prime)=f(x_i, x_{i-1})$, $G_i(x_1, x_2, \cdots,x_k, t^\prime)=g(x_i, x_{i-1})$ for $i=2, 3, \cdots,k$, $\{\cdot\}^T$ represents the transpose of the matrix '$\cdot$',  $\nabla_{x_i}=(\frac{\partial}{\partial x_{i1}},\frac{\partial}{\partial x_{i2}},\cdots,\frac{\partial}{\partial x_{id}})^T$ with $x_{ij}$ being the $j$-th component of the vector $x_i$ ($j=1, 2, \cdots,d$),  and ${\rm {Tr}}\{\cdot\}$ represents the trace of the matrix '$\cdot$'.

Therefore, the following corollary follows from  theorem \ref{theorem1} and remark \ref{remark3}.
\begin{coro}[Corollary of the main result]
Suppose  $f$ and $g$ are Lipshitz continuous, $f\in C^1(\mathbb{R}^{2\times d}, \mathbb{R}^d)$, $g\in C^2(\mathbb{R}^{2\times d}, \mathbb{R}^{d\times n})$,  $\gamma\in C^1(\mathbb{R}, \mathbb{R}^d)$  and $\forall  x,y \in \mathbb{R}^d$,  there exists a positive constant $\epsilon$ such that  $g(x,y)\ge \epsilon$. Then   $\forall t >0$,  the probability density function $\mathcal{P}_\mathcal{A}(x,t)$  for the solution $X(t)$ to  SDDE (\ref{s1_1}) exists.   Moreover, 
\begin{align} \label{s1_15}
{\mathcal P}_{\mathcal{A}}(x,t)=  
\begin{cases}
 \mathcal{Q}_1(  x; t\big|  \gamma_0;0), \quad\text{for} \quad  t\in (0, \tau],\\ \\  \\
 \displaystyle\int _{\mathbb{R}^{2\times (k-1)\times d}}\left[ \mathcal{Q}_{k-1}( y_1, y_2, \cdots, y_{k-1}; \tau\big|   x_1, x_2, \cdots, x_{k-1}; t-(k-1)\tau)\right. \\ 
  \quad \left.\times \mathcal{Q}_k ( x_1, x_2, \cdots, x_{k-1}, x; t-(k-1)\tau \big|  \gamma_0, y_1, y_2, \cdots, y_{k-1}; 0)\right] \prod\limits_{i=1}^{k-1}{\rm d}x_i\prod\limits_{i=1}^{k-1}{\rm d}y_i, \\
  \quad\text{for }\quad  t\in((k-1)\tau, k\tau]\quad\text{with}\quad k\in\mathbb{N}, k\ge 2,
\end{cases}
\end{align}
\normalsize
where $\mathcal{Q}_k$ satisfies PDE (\ref{s1_14}).
\end{coro}

\section{Proof of Theorem \ref{theorem1}}
We first present some lemmas which will be used in proof of Theorem \ref{theorem1}.

For each solution to SDDE (\ref{s1_1}), we can uniquely construct a solution to SDE (2) with continuous condition. In fact,  construct $X_i(t^\prime)$ by  $X_i(t^\prime)=X(t^\prime+(i-1)\tau)$ for $0\le t'\le \tau$ and $i=1,2,\cdots,k$, then the path of $X_i(t^\prime)$ in the time interval $t^\prime \in [0,\tau]$ is coincident with that of $X(t)$  in the time interval $t\in [(i-1)\tau, i\tau]$. It is straightforward to check that $X_i(t^\prime)$ ($i=1,2,\cdots,k$) satisfies  SDE  (\ref{s1_2}) and the continuous condition (\ref{s1_5}) (with the constant $x_0$ in (\ref{s1_5}) now becomes $\gamma_0$). On the other hand, for each solution to the SDE defined by (\ref{s1_2})  and (\ref{s1_5}) (with $x_0=\gamma_0$ in (\ref{s1_5})), we can uniquely construct a solution to SDDE (\ref{s1_2}) by     $X(t)=X_i(t-(i-1)\tau)$ for $t\in [(i-1)\tau, i\tau]$ and $i=1, 2, \cdots,k$. Therefore,  we get the following lemma.

\begin{lemma}\label{lemma1}
 SDDE  (\ref{s1_1}) in the time span $t\in [0, k\tau]$ with $k\in \mathbb{N}$ has strong unique solution if and only if the SDE defined by  (\ref{s1_2}) and (\ref{s1_5})  with $x_0=\gamma_0$  has unique strong solution. Moreover, the two solutions are related  by
\begin{align}\label{s1_16}
X(t)\overset{a.s.}{=}
\begin{cases}
X_1(t)  &\text{for} \quad t\in [0,\tau],\\
X_2(t-\tau)  &\text{for} \quad t\in [\tau,2\tau],\\
\cdots\\
X_{k}(t-(k-1)\tau) &\text{for} \quad t\in [(k-1)\tau,k\tau],\\
\end{cases}
\end{align}
or equivalently
\begin{align}\label{s1_17}
X_i(t^\prime)\overset{a.s.}{=} X(t^\prime+(i-1)\tau)\quad \text{for}\quad t^\prime\in [0,\tau] \quad \text{and}\quad i=1, 2, \cdots, k.
\end{align}
\end{lemma}

\begin{lemma}\label{lemma2} Suppose Assumption $H1$ holds. Then
$\forall k\in \mathbb{N}$, $t^\prime\in (0, \tau)$ and $x_i,y_i,z_i\in\mathbb{R}^d$ ($i=1,2,\cdots,k$),  the density for $\mathbb{R}^{k\times d}$-valued solution $\widetilde{\mathbf{X}}_k(t^\prime)$ at  $  (x_1, x_2, \cdots, x_k) $  defined by SDE (\ref{s1_2}) with bridge condition $\widetilde{\mathbf{X}}_k(0) = (y_1, y_2, \cdots, y_k)$ and $\widetilde{\mathbf{X}}_k(\tau)=(z_1, z_2,\cdots, z_k) $ exists and can be expressed as
\begin{align}\label{s1_18}
&p(\widetilde{\mathbf{X}}_k(t^\prime) = (x_1, x_2, \cdots, x_k)  \big|\widetilde{\mathbf{X}}_k(0) = (y_1, y_2, \cdots, y_k); \widetilde{\mathbf{X}}_k(\tau) = (z_1, z_2, \cdots, z_k))\nonumber\\
& = \frac{{\mathcal Q}_k (  z_1,z_2, \cdots, z_k; \tau\big| x_1, x_2, \cdots, x_k;  t^\prime){\mathcal Q}_k (   x_1, x_2, \cdots, x_k; t^\prime \big| y_1, y_2, \cdots, y_{k}; 0)}{{\mathcal Q}_k (  z_1, z_2, \cdots, z_k; \tau \big|  y_1, y_2, \cdots, y_{k}; 0)}\nonumber\\
\end{align}
\normalsize
\end{lemma}

\begin{proof}[Proof of Lemma \ref{lemma2}]

If Assumption $H1$ holds, the   conditional density of $\widetilde{\mathbf{X}}_k (\tau)$  given both  values of $\widetilde{\mathbf{X}}_k (0)$ and $\widetilde{\mathbf{X}}_k (t^\prime)$ exists. In fact, by Markov property of SDE (\ref{s1_2}) \cite{Protter2004}, this density is exactly the same as  the   density of   $\widetilde{\mathbf{X}}_k (\tau)$  under the condition that only  the value of $\widetilde{\mathbf{X}}_k (t^\prime)$ is given, i.e., 
\begin{align}\label{s1_19}
&p(\widetilde{\mathbf{X}}_k(\tau) = (z_1, z_2, \cdots, z_k) \big| \widetilde{\mathbf{X}}_k(0) = (y_1, y_2, \cdots, y_k) ; \widetilde{\mathbf{X}}_k(t^\prime) = (x_1, x_2, \cdots, x_k) )\nonumber\\
&=p(\widetilde{\mathbf{X}}_k(\tau) = (z_1, z_2, \cdots, z_k) \big| \widetilde{\mathbf{X}}_k(t^\prime) = (x_1, x_2, \cdots, x_k)) \nonumber\\
&=\mathcal{Q}_k ( z_1, z_2, \cdots, z_k; \tau\big|  x_1, x_2, \cdots, x_k; t^\prime).
\end{align}
The identity
\begin{align}\label{s1_20}
&p(\widetilde{\mathbf{X}}_k(t^\prime) =  (x_1, x_2, \cdots, x_k) \big|\widetilde{\mathbf{X}}_k(0) = (y_1, y_2, \cdots, y_k) ; \widetilde{\mathbf{X}}_k(\tau) = (z_1, z_2, \cdots, z_k) )\nonumber\\
&  = p(\widetilde{\mathbf{X}}_k(t^\prime) = (x_1, x_2, \cdots, x_k)  \big|\widetilde{\mathbf{X}}_k(0) = (y_1, y_2, \cdots, y_k) )\nonumber\\
& \quad\times \frac{p(\widetilde{\mathbf{X}}_k(\tau) = (z_1, z_2, \cdots, z_k) \big| \widetilde{\mathbf{X}}_k(0) = (y_1, y_2, \cdots, y_k) ; \widetilde{\mathbf{X}}_k(t^\prime) = (x_1, x_2, \cdots, x_k) )}{p(\widetilde{\mathbf{X}}_k(\tau) = (z_1, z_2, \cdots, z_k) \big| \widetilde{\mathbf{X}}_k(0) = (y_1, y_2, \cdots, y_k) )}
\end{align}
indicates that  the density for  $\widetilde{\mathbf{X}}_k(t^\prime)$ defined by (\ref{s1_2}) under the  bridge condition $\widetilde{\mathbf{X}}_k(0) = (y_1, y_2, \cdots, y_k)$ and $\widetilde{\mathbf{X}}_k(\tau)=(z_1, z_2,\cdots, z_k)$ exists since the right hand side of (\ref{s1_20}) is well defined  by Assumption $H1$.

Substitute (\ref{s1_19}) into (\ref{s1_20}), and change the notation $p$ to $\mathcal{Q}_k$ (e.g., see (\ref{s1_9})), we get (\ref{s1_18}).
\end{proof}

\begin{lemma}\label{lemma_r1} 
$\forall k\in \mathbb{N}, x_i \in \mathbb{R}^d$ ($i=1,2, \cdots,k$) and $\tau\ge t^\prime>s\ge 0$,    the following relationship between $\mathcal{Q}_{k+1}$ and $\mathcal{Q}_{k}$ is true,
\begin{align}\label{eq_r1}
\mathcal{Q}_k(x_1, x_2, \cdots,x_k; t^\prime\big|y_1, y_2, \cdots, y_k;s)  = \int_{\mathbb{R}^d} \mathcal{Q}_{k+1}(x_1, x_2, \cdots,x_k,x_{k+1}; t^\prime\big|y_1,y_2,\cdots,y_k, y_{k+1};s) {\rm d}x_{k+1}
\end{align}
\end{lemma}
\begin{proof}[Proof of Lemma \ref{lemma_r1}]
Note that
\begin{align}\label{eq_r2}
&\int_{\mathbb{R}^d} \mathcal{Q}_{k+1}(x_1, x_2, \cdots, x_{k+1}; t^\prime\big|y_1, y_2, \cdots ,y_{k+1};s){\rm d}x_{k+1}\nonumber\\
&=\int_{\mathbb{R}^d} p(\widetilde{\mathbf{X}}_{k+1}(t^\prime) = (x_1, x_2, \cdots, x_{k+1})\big|\widetilde{\mathbf{X}}_{k+1}(s)=(y_1,y_2,\cdots,y_{k+1})){\rm d}x_{k+1}\nonumber\\
&= p(\widetilde{\mathbf{X}}_{k}(t^\prime) = (x_1, x_2, \cdots, x_{k})\big|\widetilde{\mathbf{X}}_{k+1}(s)=(y_1,y_2,\cdots,y_{k+1})). \nonumber\\
\end{align}
It follows from SDE (\ref{s1_2}) that the value of  $\widetilde{\mathbf{X}}(t^\prime)$ depends only on its initial value  $\widetilde{\mathbf{X}}_{k}(s)$ and independent of the value of  $X_{k+1}(s)$, i.e., 
\begin{align}\label{eq_r3}
&p(\widetilde{\mathbf{X}}_{k}(t^\prime) = (x_1, x_2, \cdots, x_{k})\big|\widetilde{\mathbf{X}}_{k+1}(s)=(y_1,y_2,\cdots,y_{k+1}))\nonumber\\
& = p(\widetilde{\mathbf{X}}_{k}(t^\prime) = (x_1, x_2, \cdots, x_{k})\big|\widetilde{\mathbf{X}}_{k}(s)=(y_1,y_2,\cdots,y_{k}))\nonumber\\
&=\mathcal{Q}_{k}(x_1, x_2, \cdots, x_{k}; t^\prime\big|y_1, y_2, \cdots ,y_{k};s).
\end{align}
Then by using (\ref{eq_r2}) and (\ref{eq_r3}), we get (\ref{eq_r1}).
\end{proof}

\begin{lemma}\label{lemma3}
Supposse Assumption $H1$ holds. Then the following relationship between $\mathcal{P}_\mathcal{A}$, the density for SDDE (\ref{s1_1}), and $\mathcal{Q}_\mathcal{A}$, the transitional density for the SDE  (\ref{s1_2}),  is true.\\
\\
(i) \\ $\forall t\in (0, \tau]$ and $y \in \mathbb{R}^d$,
\begin{align}\label{s1_21}
\mathcal{P}_\mathcal{A}(y, t)= Q_1(y;t\big| \gamma_0; 0).
\end{align}
(ii) \\
 $\forall k\ge 2$, $k\in \mathbb{N}$ and $ x_i\in \mathbb{R}^d$ ($i=1,2,\cdots,k$)
\begin{align}\label{s1_22}
{\mathcal P}_{\mathcal{A}} (x_k, k\tau \big| x_1, \tau; x_2, 2\tau; \cdots; x_{k-1}, (k-1)\tau) = \frac{\mathcal{Q}_k(x_1, x_2, \cdots,x_k; \tau\big|\gamma_0, x_1, \cdots, x_{k-1}; 0)}{\mathcal{Q}_{k-1}( x_1, x_2, \cdots,x_{k-1};\tau\big| \gamma_0, x_1, \cdots, x_{k-2}; 0)}.
\end{align}
(iii) \\ $\forall k\ge 2, k\in \mathbb{N}$ and $x_i\in \mathbb{R}^d$ ($i=1,2,\cdots,k$),
\begin{align}\label{s1_23}
{\mathcal P}_{\mathcal{A}}(x_1, \tau; x_2, 2\tau; \cdots; x_k, k\tau)= \mathcal{Q}_k( x_1, x_2, \cdots,x_k; \tau\big| \gamma_0, x_1, \cdots, x_{k-1}; 0).
\end{align}
(iv)
\\ $\forall k\ge 2, k \in \mathbb{N}$ and $t\in ((k-1)\tau, k\tau)$, 
\begin{align}\label{s1_24}
&{\mathcal P}_{\mathcal{A}}(y,t \big| x_1, \tau; x_2, 2\tau; \cdots; x_{k-1}, (k-1)\tau)\nonumber\\
&= \displaystyle{\int} _{\mathbb{R}^{(k-1)\times d}} {\mathcal Q}_{k-1} (  x_1,x_2, \cdots, x_{k-1}; \tau\big| y_1, y_2, \cdots, y_{k-1}; t-(k-1)\tau) \nonumber\\
&\quad\quad\times  \dfrac{{\mathcal Q}_k ( y_1, y_2, \cdots, y_{k-1},y; t-(k-1) \tau\big|  \gamma_0, x_1, x_2, \cdots, x_{k-1}; 0)}{{\mathcal Q}_{k-1} (  x_1, x_2, \cdots, x_{k-1}; \tau \big| \gamma_0, x_1, x_2, \cdots, x_{k-2}; 0)}  \; {\rm d}y_1 {\rm d}y_2\cdots{\rm d}y_{k-1}
\end{align}
\end{lemma}

\begin{proof}[Proof of Lemma \ref{lemma3}]

First, we prove (i).

 By Lemma \ref{lemma1}, the density of $X(t)$ for $0<t\le \tau$ defined by (\ref{s1_1}) is  the same as the density of $X_1(t)$   defined by (2)  under the condition that $X_1(0)=\gamma_0$. Therefore, (i) is true.

Now, we prove (ii).

  By Lemma \ref{lemma1},  the density of $X(k\tau)$ under the condition that $X(\tau)=x_1, X(2\tau)=x_2, \cdots, X_{k-1}((k-1)\tau)=x_{k-1}$ is exactly the same as the density of $X_k(\tau)$ defined by (\ref{s1_2})  under the continuous  condition that $X_1(0)=\gamma_0, X_1(\tau)=X_2(0)=x_1, \cdots, X_{k-1}(\tau)=X_{k}(0)=x_{k-1}$, i.e.,
\begin{align}\label{s1_25}
&{\mathcal P}_{\mathcal{A}} (x_k, k\tau \big| x_1, \tau; x_2, 2\tau; \cdots; x_{k-1}, (k-1)\tau)\nonumber\\
&=p(X(k\tau)=x_k \big|X(0)=\gamma_0; X(\tau)=x_1; X(2\tau)=x_2; \cdots; X((k-1)\tau)=x_{k-1})\nonumber\\
& =p(X_k(\tau)=x_k\big| X_1(0)=\gamma_0, X_1(\tau)=X_2(0)=x_1, \cdots, X_{k-1}(\tau)=X_{k}(0)=x_{k-1})\nonumber\\
&  =p(X_k(\tau)=x_k\big| X_1(0)=\gamma_0; X_2(0)=x_1; \cdots; X_{k}(0)=x_{k-1};  \nonumber\\
&\quad\quad\quad\quad\quad\quad X_1(\tau)=x_1; X_2(\tau)=x_2; \cdots; X_{k-1}(\tau)=x_{k-1})\nonumber\\
&=p(X_k(\tau)=x_k\big|\widetilde{\mathbf{X}}_{k}(0) = (\gamma_0, x_1, x_2, \cdots,x_{k-1}) ; \widetilde{\mathbf{X}}_{k-1}(\tau)=(x_1, x_2, \cdots,x_{k-1}) )\nonumber\\
&=\frac{p(\widetilde{\mathbf{X}}_{k-1}(\tau)=(x_1, x_2, \cdots, x_{k-1}) ;X_k(\tau)=x_k \big|\widetilde{\mathbf{X}}_{k}(0) = (\gamma_0, x_1, x_2, \cdots,x_{k-1}) )}{p(\widetilde{\mathbf{X}}_{k-1}(\tau)=(x_1, x_2, \cdots, x_{k-1}) \big|\widetilde{\mathbf{X}}_{k}(0) = (\gamma_0, x_1, x_2, \cdots,x_{k-1}) )}\nonumber\\
&=\frac{p(\widetilde{\mathbf{X}}_{k}(\tau)=(x_1, x_2, \cdots, x_{k})  \big|\widetilde{\mathbf{X}}_{k}(0) = (\gamma_0, x_1, x_2, \cdots,x_{k-1}) )}{p(\widetilde{\mathbf{X}}_{k-1}(\tau)=(x_1, x_2, \cdots, x_{k-1}) \big|\widetilde{\mathbf{X}}_{k}(0) = (\gamma_0, x_1, x_2, \cdots,x_{k-1}) )}\nonumber\\
&=\frac{p(\widetilde{\mathbf{X}}_{k}(\tau)=(x_1, x_2, \cdots, x_{k})  \big|\widetilde{\mathbf{X}}_{k}(0) = (\gamma_0, x_1, x_2, \cdots,x_{k-1}) )}{p(\widetilde{\mathbf{X}}_{k-1}(\tau)=(x_1, x_2, \cdots, x_{k-1})\big|\widetilde{\mathbf{X}}_{k-1}(0) = (\gamma_0, x_1, x_2, \cdots,x_{k-2} ))}.
\end{align}
The last identity in (\ref{s1_25}) follows from the fact that
 $\widetilde{\mathbf{X}}_{k-1}(\tau)$ is determined only by $\widetilde{\mathbf{X}}_{k-1}(0)$ and independent of $X_{k-1}(0)$, as can be straightforwardly checked   with SDE (\ref{s1_2}).

By Assumption $H1$, the right hand side of (\ref{s1_25}) is well defined. Change the notation   $p$ to $\mathcal{Q}$ (e.g., see (\ref{s1_9})), (\ref{s1_25}) becomes (\ref{s1_22}).

Now, we prove (iii).

Note that
\begin{align}\label{s1_26}
&{\mathcal P}_{\mathcal{A}}(x_1, \tau; x_2, 2\tau; \cdots; x_k, k\tau)\nonumber\\
&={\mathcal P}_{\mathcal{A}}(x_1, \tau)\times{\mathcal P}_{\mathcal{A}}(x_2, 2\tau\big| x_1, \tau)\;{\mathcal P}_{\mathcal{A}}(x_3, 3\tau\big| x_1, \tau; x_2, 2\tau)\times\cdots\nonumber\\
&\quad\quad\times{\mathcal P}_{\mathcal{A}} (x_k, k\tau \big| x_1, \tau; x_2, 2\tau; \cdots; x_{k-1}, (k-1)\tau),
\end{align}
Substitute (\ref{s1_21}) and (\ref{s1_22}) into (\ref{s1_26}), we get (\ref{s1_23}).

In the following, we prove (iv).

Note that
\begin{align}\label{s1_27}
&{\mathcal P}_{\mathcal{A}} (y, t\big| x_1, \tau; x_2, 2\tau; \cdots; x_{k-1}, (k-1)\tau)\nonumber\\
&=\int_{\mathbb{R}^d}\left[{\mathcal P}_{\mathcal{A}} (y, t\big| x_1, \tau; x_2, 2\tau; \cdots; x_{k-1}, (k-1)\tau; x_k, k\tau)\times {\mathcal P}_{\mathcal{A}} (x_k, k\tau\big| x_1, \tau; x_2, 2\tau; \cdots; x_{k-1}, (k-1)\tau)\right]{\rm d}x_k
\nonumber\\
&=\int_{\mathbb{R}^{k\times d}}\left[ {\mathcal P}_{\mathcal{A}} (y_1, t-(k-1)\tau; y_2, t-(k-2)\tau; \cdots; y_{k-1}, t-\tau; y, t\big| x_1, \tau; x_2, 2\tau; \cdots; x_{k-1}, (k-1)\tau; x_k, k\tau)\right.\nonumber\\
&\left.\quad\quad\quad\quad\quad\quad\times {\mathcal P}_{\mathcal{A}} (x_k, k\tau\big| x_1, \tau; x_2, 2\tau; \cdots; x_{k-1}, (k-1)\tau)\right]{\rm d}x_k \prod\limits_{i=1}^{k-1} {\rm d}y_i.
\end{align}

By   Lemma \ref{lemma1}, the joint density for $X(t-(k-1)\tau),X(t-(k-2)\tau),\cdots,X(t)$ under the condition that $X(\tau)=x_1, X(2\tau)=x_2,\cdots,X(k\tau)=x_{k}$ is the same as the joint density for $X_1(t-(k-1)\tau), X_2(t-(k-1)\tau),\cdots,X_k(t-(k-1)\tau)$ defined by SDE (\ref{s1_2}) with continuous condition $X_1(0)=\gamma_0, X_1(\tau)=X_2(0)=x_1, \cdots, X_{k-1}(\tau)=X_k(0)=x_{k-1}, X_k(\tau)=x_k$, i.e., 
\begin{align}\label{s1_28}
& {\mathcal P}_{\mathcal{A}} (y_1, t-(k-1)\tau; y_2, t-(k-2)\tau; \cdots; y_{k-1}, t-\tau; y, t\big| x_1, \tau; x_2, 2\tau; \cdots; x_{k-1}, (k-1)\tau; x_k, k\tau) \nonumber\\
&= p(\widetilde{\mathbf{X}}_k(t-(k-1)\tau)=(y_1, y_2, \cdots,y_{k-1},y) \big| \widetilde{\mathbf{X}}_k(0)=(\gamma_0, x_1, x_2, \cdots, x_{k-1}) ; \widetilde{\mathbf{X}}_k(\tau)=(x_1, x_2, \cdots,x_{k-1},x_k))\nonumber\\
\end{align}
The density in (\ref{s1_28})  exists by   Lemma  \ref{lemma2}. Substitute (\ref{s1_18}) into  (\ref{s1_28}), we get  
\begin{align}\label{s1_29}
& {\mathcal P}_{\mathcal{A}} (y_1, t-(k-1)\tau; y_2, t-(k-2)\tau; \cdots; y_{k-1}, t-\tau; y, t\big| x_1, \tau; x_2, 2\tau; \cdots; x_{k-1}, (k-1)\tau;x_k, k\tau) \nonumber\\
&={\mathcal Q}_k  (y_1, y_2, \cdots, y_{k-1},y; t-(k-1)\tau \big|  \gamma_0, x_1, x_2, \cdots, x_{k-1}; 0)\nonumber\\
&\quad\times \frac{{\mathcal Q}_k (  x_1,x_2, \cdots, x_k; \tau\big| y_1, y_2, \cdots, y_{k-1},y; t-(k-1)\tau)}{{\mathcal Q}_k (  x_1, x_2, \cdots, x_k; \tau \big| \gamma_0, x_1, x_2, \cdots, x_{k-1}; 0)}.
\end{align}
Substitute (\ref{s1_29}) and  (\ref{s1_22}) into (\ref{s1_27}), we get
\begin{align}\label{eq_r4}
&{\mathcal P}_{\mathcal{A}}(y,t \big| x_1, \tau; x_2, 2\tau; \cdots; x_{k-1}, (k-1)\tau)\nonumber\\
&= \displaystyle{\int} _{\mathbb{R}^{k\times d}} {\mathcal Q}_{k} (  x_1,x_2, \cdots, x_{k}; \tau\big| y_1, y_2, \cdots, y_{k-1}, y; t-(k-1)\tau) \nonumber\\
&\quad\quad\times  \dfrac{{\mathcal Q}_k ( y_1, y_2, \cdots, y_{k-1},y; t-(k-1) \tau\big|  \gamma_0, x_1, x_2, \cdots, x_{k-1}; 0)}{{\mathcal Q}_{k-1} (  x_1, x_2, \cdots, x_{k-1}; \tau \big| \gamma_0, x_1, x_2, \cdots, x_{k-2}; 0)}  \;{\rm d}x_k {\rm d}y_1 {\rm d}y_2\cdots{\rm d}y_{k-1}
\end{align}
By using Lemma \ref{lemma_r1}, (\ref{eq_r4}) becomes (\ref{s1_24}).
\end{proof}

We are now ready to present the proof of the main result.
\begin{proof}[\bf{Proof of Theorem \ref{theorem1}:}] 

The statement (i) follows immediately from Lemma \ref{lemma3} (i).

Now we prove (ii).

 For $t= k\tau$ with $k\in \mathbb{N}$ and $k\ge 2$, it follows from Lemma \ref{lemma3}(iii) that  the joint  density ${\mathcal P}_{\mathcal A} (x_1, \tau; x_2, 2\tau; \cdots; x_{k-1}, (k-1)\tau; x, k\tau)$ exists. Then  ${\mathcal P}_{\mathcal A} (x, k\tau)$   exists by the  identity  
\begin{align}\label{s1_30}
{\mathcal P}_{\mathcal A} (x, k\tau) = \int_{{\mathbb{R}}^{(k-1)\times d}} {\mathcal P}_{\mathcal A} (x_1, \tau; x_2,  2\tau; \cdots; x_{k-1}, (k-1)\tau; x, k\tau) \prod\limits_{i=1}^{k-1}{\rm d}x_i\;.
\end{align}
Substitute (\ref{s1_23})  into (\ref{s1_30}), we get (\ref{s1_11}).

In the following, we prove (iii). 

For $t\in ((k-1), k\tau)$ with $k\ge 2$ and $k\in \mathbb{N}$,  it follows from Lemma \ref{lemma3}(ii) and Lemma \ref{lemma3}(iii) that both densities $\mathcal{P}_\mathcal{A} (x,t\big| x_1, \tau; x_2, \tau; \cdots; x_{k-1}, (k-1)\tau)$  and $\mathcal{P}_\mathcal{A}(x_1, \tau; x_2, \tau; \cdots; x_{k-1}, (k-1)\tau)$ exist.  Then $\mathcal{P}_\mathcal{A} (x,t)$ exists by the identity
\begin{align}\label{s1_31}
&\mathcal{P}_\mathcal{A} (x,t)\nonumber\\
&=\int_{\mathbb{R}^{(k-1)\times d}} \left[{\mathcal P}_{\mathcal{A}} (x, t\big| x_1, \tau; x_2, 2\tau; \cdots; x_{k-1}, (k-1)\tau)\times{\mathcal P}_{\mathcal{A}} ( x_1, \tau; x_2, 2\tau; \cdots; x_{k-1}, (k-1)\tau)\right]\prod\limits_{i=1}^{k-1}{\rm d}x_i\nonumber\\
\end{align}
By using (\ref{s1_23}) and (\ref{s1_24}), (\ref{s1_31}) becomes (\ref{s1_12}).
\end{proof}

\section{Examples}
In this section,  the main result is verified by applying it to   some SDDE with known density.\\
Consider SDDE
\begin{align}\label{s4_1}
\begin{cases}
{\rm d}X(t) = X(t-1){\rm d}t + {\rm d}B(t), \quad \text{for} \quad 0<t\le2,\\
X(t)=0, \quad \text{for} \quad -1 \le t\le 0,
\end{cases}
\end{align}
where $X(t)$ is a $\mathbb{R}$-valued process, and $B(t)$ is the standard scalar Brownian motion. Note that (\ref{s4_1}) is corresponding to   (\ref{s1_1}) with $d=1$, $n=1$,  $\tau=1$, $f(x,y)=y$, $g(x,y)=1$ and $\gamma=0$.

It is easy to check that the solution of (\ref{s4_1}) is
\begin{align}\label{s4_2}
X(t)=
\begin{cases}
B(t),& \text{for} \quad t\in[0, 1],\\
(t-1)B(t-1)+ B(t) -  \int_0^{t-1} s{\rm d} B(s),& \text{for} \quad t\in[1, 2],
\end{cases}
\end{align}
and the probability density for the solution $X(t)$ is
\begin{align}\label{s4_3}
\mathcal {P}_\mathcal{A}(x,t)=
\begin{cases}
\dfrac{1}{\sqrt {2\pi t}}\exp\left\{-\dfrac{x^2}{2t}\right\}& \text{for} \quad t\in[0, 1],\\
 \dfrac{1}{\sqrt{2\pi ((t^3+2)/3)}}\exp\left\{-\dfrac{3x^2}{2(t^3+2)}\right\},& \text{for} \quad t\in[1, 2],
\end{cases}
\end{align}

In the following, the corollary of the main theorem  is applied to solve  the density for SDDE (\ref{s1_1}) in the time span $t\in (0,2]$. The obtained density is then compared  with the exact result in (\ref{s4_3}).

By the corollary of the main result, the density of $X(t)$ can be expressed as
\begin{align}\label{s4_4}
\mathcal {P}_\mathcal{A}(x,t) =  \mathcal{Q}_1( x; t\big| 0;0), \quad\text{for}\quad  t\in (0,1],
\end{align}
and 
\begin{align}\label{s4_5}
\mathcal {P}_\mathcal{A}(x,t)=\displaystyle\int_{\mathbb{R}^{2}}  \mathcal{Q}_1( y_1; 1\big| x_1; t-1)\times\mathcal{Q}_2( x_1, x; t-1\big| 0,y_1;0)  {\rm d}x_1\,{\rm d}y_1, \quad\text{for}\quad t\in [1,2],
\end{align}
where $\mathcal{Q}_1(x;t\big| y;s)$   satisfies the PDE
\begin{align}\label{s4_6}
\begin{cases}
\dfrac{\partial}{\partial t}\mathcal{Q}_{1}(x;t\big| y;s)=  \dfrac{1}{2}\dfrac{\partial^2}{\partial x ^2}  \mathcal{Q}_{1}(x;t\big| y;s),\quad \text{for}\quad t>s,\\
 \lim\limits_{t\to s}\mathcal{Q}_1( x;t\big| y;s)= \delta(x-y),
\end{cases}
\end{align}
and $\mathcal{Q}_2( x_1, x_2; t \big|  y_1, y_2; s)$ with $0\le s < t\le \tau$ satisfies the PDE
\begin{align} \label{s4_7}
\begin{cases}
&\dfrac{\partial}{\partial t}\mathcal{Q}_{2} ( x_1, x_2; t \big|  y_1, y_2; s)= -x_1\dfrac{\partial}{\partial x_2}  \mathcal{Q}_{2} ( x_1, x_2; t \big|  y_1, y_2; s ) + \dfrac{1}{2}\dfrac{\partial^2}{\partial x_1^2}  \mathcal{Q}_{2}( x_1, x_2; t \big|  y_1, y_2; s )  \\
&\quad\quad\quad  + \dfrac{1}{2}\dfrac{\partial^2}{\partial x_{2}^2} \mathcal{Q}_{2}( x_1, x_2; t \big|  y_1, y_2; s ) , \quad\text{for}\quad t>s, \\ \\
&\lim\limits_{t\to s}\mathcal{Q}_2( x_1, x_2; t \big|  y_1, y_2; s) = \delta(x_1-y_1) \delta(x_2- y_2).
\end{cases}
\end{align}
It is straightforward to check that the solutions to PDEs (\ref{s4_6}) and (\ref{s4_7}) can be respectively expressed as
\begin{align}\label{s4_8}
\mathcal{Q}_1( x;t\big| y;s)=\dfrac{1}{\sqrt {2\pi (t-s)}}\exp\left\{-\dfrac{(x-y)^2}{2(t-s)}\right\} 
\end{align}
and
\begin{align}\label{s4_9}
  & \mathcal{Q}_{2}( x_1,x_2;t \big| y_1,y_2;s)   = \dfrac{1}{2\pi(t-s)\sqrt{\dfrac{(t-s)^2+12}{12}}}\nonumber\\
 &\quad\quad\times \exp\left\{-\dfrac{2(t-s)^2+6}{(t-s)^2+12}\left[\dfrac{(y_1-x_1)^2}{t-s} 
   -  \dfrac{3(x_1-y_1)(x_2-y_2 -y_1(t-s))}{(t-s)^2+3} + \dfrac{3(x_2- y_2 - y_1(t-s))^2}{  (t-s)^3+3(t-s) }\right]\right\}. 
\end{align}
It follows (\ref{s4_8}) that
\begin{align}\label{s4_10}
\mathcal{Q}_1(x;t\big| 0;0)=\dfrac{1}{\sqrt {2\pi  t }}\exp\left\{-\dfrac{ x ^2}{2t }\right\}
\end{align}
and
\begin{align}\label{s4_12}
\mathcal{Q}_1(y_1;1\big| x_1;t-1)=\dfrac{1}{\sqrt {2\pi (2-t) }}\exp\left\{-\dfrac{ (y_1-x_1) ^2}{2(2-t) }\right\}.
\end{align}
It follows from (\ref{s4_9}) that
\begin{align}\label{s4_11}
   & \mathcal{Q}_{2}(x_1,x;t-1 \big|0,y_1;0 ) \notag\\
   =& \frac{1}{2\pi(t-1)\sqrt{\frac{(t-1)^2+12}{12}}} \exp\left\{-\frac{2(t-1)^2+6}{(t-1)^2+12}\left[\frac{x_1^2}{t-1}
   - \frac{3x_1(x-y_1)}{(t-1)^2+3}
   + \frac{3(x-y_1)^2}{(t-1)((t-1)^2+3)}\right]\right\}.
\end{align}
By using (\ref{s4_11}) and (\ref{s4_12}), we get 
\begin{align}\label{s4_13}
 & \int_{\mathbb{R}^2}\mathcal{Q}_{1}(y_1;1 \big|x_1;t-1)\mathcal{Q}_{2}( x_1,x, t-1 \big|0, y_1;0 ) dy_1dx_1\notag\\
  &=\int_{\mathbb{R}^2}\frac{1}{\sqrt{2\pi(2-t)}}\exp\left\{-\frac{(y_1-x_1)^2}{2(2-t)}\right\}\notag\\
  &\quad\times \frac{1}{2\pi(t-1)\sqrt{\frac{(t-1)^2+12}{12}}} \exp\left\{-\frac{2(t-1)^2+6}{(t-1)^2+12}\left(\frac{x_1^2}{t-1}
   - \frac{3x_1(x-y_1)}{(t-1)^2+3}
   + \frac{3(x-y_1)^2}{(t-1)((t-1)^2+3)}\right)\right\}dy_1dx_1.\notag\\
   &=\int_{\mathbb{R}^2}\frac{1}{2\pi(t-1)\sqrt{\frac{(t-1)^2+12}{12}}\sqrt{2\pi(2-t)}}\notag\\
   &\quad\times\exp\left\{-\left(\frac{1}{2(2-t)}+\frac{6}{12(t-1)+(t-1)^3}\right)y_1^2
   +\left(\frac{-6x_1}{12 +(t-1)^2}+\frac{x_1}{2-t}+\frac{12x}{12(t-1)+(t-1)^3}\right)y_1\right\}\notag\\
   &\quad\times\exp\left\{\frac{-6x^2}{(t-1)^3+12(t-1)}+\frac{6xx_1}{(t-1)^2+12}
   -\frac{x_1^2}{2(2-t)}-\frac{(2(t-1)^2+6)x_1^2}{ (t-1)^3+12(t-1)}\right\}dx_1dy_1\notag\\
  &=\int_{\mathbb{R} }\frac{\sqrt{3}}{\pi\sqrt{(t-1)^4+12(t-1)}}\times \exp\left\{\frac{-6x^2}{(t-1)^3+12 }\right\}\notag\\
  &\quad\times \exp\left\{-\left(\frac{2(t-1)^3+6(t-1)^2+6t}{(t-1)^4+12(t-1) }\right)x_1^2
  +\left(\frac{6(t+1)x}{(t-1)^3+12 }\right)x_1\right\}dx_1\notag\\
  &= \frac{1}{\sqrt{2\pi ((t^3+2)/3)}}\exp\left\{-\frac{3x^2}{2(t^3+2)}\right\}.
\end{align}
Substitute (\ref{s4_10}) into (\ref{s4_4}), and (\ref{s4_13}) into (\ref{s4_5})  respectively,  we get  the desired result (\ref{s4_3}).

\section*{Acknowledgment}
This work was supported by the National Natural Science Foundation of China
grants 11531006.


\end{document}